\documentclass{amsart}

\usepackage[table,xcdraw]{xcolor}
\usepackage{amssymb}
\usepackage{amsmath}
\usepackage{amsfonts}
\usepackage{amsxtra}
\usepackage{amsthm}
\usepackage{tikz}
\usepackage{tikz-cd}
\usepackage{float}
\usepackage{booktabs}
\usepackage{color}
\usepackage{xcolor}
\usepackage{graphicx}
\usepackage{enumitem}
\usepackage{venndiagram}
\usetikzlibrary{positioning,shapes,fit,arrows}
\usepackage{multirow}
\usepackage{multicol}
\usepackage{url}
\usepackage{hyperref}
\usepackage{listings}

\newtheorem{theorem}{Theorem}[section]
\newtheorem{proposition}[theorem]{Proposition}
\newtheorem{corollary}[theorem]{Corollary}
\newtheorem{lemma}[theorem]{Lemma}
\newtheorem{exmp}[theorem]{Example}
\newtheorem{definition}[theorem]{Definition}

\newcommand{\abf}{\mathbf{a}}

\newcommand{\Id}{\mathrm{Id}}
\newcommand{\Ima}{\mathrm{Im}}

\newcommand{\homr}{\mathrm{Hom}_{\mathrm{Ring}}}

\newcommand{\mM}{\mathcal{M}}

\def\f{\mathbb{F}}

\def\n{\mathbb{N}}
\def\z{\mathbb{Z}}

\def\inv{^{-1}}

\definecolor{codegreen}{rgb}{0,0.6,0}
\definecolor{codegray}{rgb}{0.5,0.5,0.5}
\definecolor{codepurple}{rgb}{0.58,0,0.82}
\definecolor{backcolour}{rgb}{0.95,0.95,0.92}

\lstdefinestyle{mystyle}{
    backgroundcolor=\color{backcolour},   
    commentstyle=\color{codegreen},
    keywordstyle=\color{magenta},
    numberstyle=\tiny\color{codegray},
    stringstyle=\color{codepurple},
    basicstyle=\ttfamily\footnotesize,
    breakatwhitespace=false,         
    breaklines=true,                 
    captionpos=b,                    
    keepspaces=true,                 
    numbers=left,                    
    numbersep=5pt,                  
    showspaces=false,                
    showstringspaces=false,
    showtabs=false,                  
    tabsize=2
}

\begin{document}

\title[]{Towards an Enumeration of Finite Common Meadows}

\author[]{João Dias}
\author[]{Bruno Dinis}

\address[Bruno Dinis]{Departamento de Matemática, Universidade de Évora}
\email{bruno.dinis@uevora.pt}

\address[João Dias]{Departamento de Matemática, Universidade de Évora}
\email{joao.miguel.dias@uevora.pt}



\begin{abstract}
Common meadows are commutative and associative algebraic structures with two operations (addition and multiplication) with additive and multiplicative identities and for which inverses are total. The inverse of zero is an error term $\abf$ which is absorbent for addition. We study the problem of enumerating all finite common meadows of \emph{order} $n$ (that is, common meadows with $n$ elements). This problem turns out to be deeply connected with both the number of finite rings of order $n$ and with the number of a certain kind of partition of positive integers.
\end{abstract}

\maketitle

\section{Introduction}

Meadows are algebraic structures with two operations (addition and multiplication), introduced by Bergstra and Tucker in \cite{Bergstra2006}. Perhaps the most interesting feature of these algebraic structures, which was one of the main motivations to introduce and study them, is that they allow to divide by zero, i.e.\ both addition and multiplication are operations for which the inverses are total. In particular, common meadows (introduced by Bergstra and Ponse in \cite{Bergstra2015}), which can be decomposed as disjoint unions of rings \cite{Dias_Dinis(23)}, enable to invert zero by introducing a term $\abf$, such that $0\inv=\abf$.

For the most part, meadows have been studied as abstract data types given by equational axiomatizations \cite{Bergstra2008,BP(20),10.1145/1219092.1219095,bergstra2020arithmetical}. These allow to obtain simple term rewriting systems which are easier to automate in formal reasoning  \cite{bergstra2020arithmetical,bergstra2023axioms}. However, more recently, connections with nonstandard analysis \cite{Dinis_Bottazzi} and the study of common meadows from a purely algebraic point of view \cite{Dias_Dinis(23)}, opened new lines of research in the field. The latter paper also introduced a new class of meadows called \emph{pre-meadows with $\abf$} which we consider in this paper as well.

The present paper also views meadows from an algebraic point of view but focuses on finite meadows only. We study the problem of enumerating all finite common meadows of \emph{order} $n$ (that is, common meadows with $n$ elements). This problem is, of course, related with the similar problem of enumerating other finite algebraic structures, such as semigroups 
\cite{Forsythe,smallsemi, distler2012semigroups,Distler9}, groups \cite{Blackburn_Neumann_Venkataraman_2007} and rings \cite{Blackburn}. 

For finite involutive meadows, i.e.\ such that $0 \inv =0$, it is possible to obtain a unique representation of minimal finite meadows in terms of finite prime fields \cite{bethke}, which solves the problem of enumerating finite involutive meadows.
For pre-meadows with $\abf$ we show that there exists a characterization in terms of directed lattices of rings. The case of common meadows requires an extra condition that in many cases may prove difficult to check.

 The number of pre-meadows with $\abf$ and of common meadows of both even and odd order grows exponentially, but at different rates (see Figures \ref{Table} and \ref{fig:graph}). Of course, since these structures need to have at least 3 distinct elements, it only makes sense to start enumerating them for orders $n \geq 3$. As it turns out, the problem of enumerating finite pre-meadows with $\abf$ is deeply connected with both the number of finite rings of order $n \geq 3$ and with partitions of positive integers of a certain kind (which we shall call \emph{admissible partitions}).   

After recalling some preliminary notions and results in Section~\ref{S:Prelim} we start, in Section~\ref{S:Finite_meadows}, by showing that there are common meadows of any given order $n \geq 3$. We also give a lower bound on the number of common meadows (which also holds for pre-meadows with $\abf$) of a given order. Our main result states that a partition of a given positive integer is admissible if and only if it is possible to construct a pre-meadow with $\abf$  associated with that partition, i.e.\ each number in the partition is the order of one of the rings that constitute the pre-meadow with $\abf$. This result is enough to establish that, for the orders $3,4,6$ and $8$, there is exactly one common meadow. The result also allows to explicitly construct all pre-meadows with $\abf$ of a given order by reducing the problem to the problem of finding the admissible partitions and the rings of that order. In Section~\ref{S:5_and_7} we illustrate this construction for orders $5$ and $7$. These examples are somewhat typical. So much so that we are able to present an algorithm, in GAP 4, developed by the first author, which allows to obtain all pre-meadows with $\abf$ of a given order, provided that it is known both the number of finite rings and the number of lattices of lower orders. In Section~\ref{S:algorithm} we describe the algorithm and present some of its output. Some final remarks are left for Section~\ref{S:finalremarks}.

\section{Preliminaries}\label{S:Prelim}

In this section we present some definitions and results, mostly from \cite{Dias_Dinis(23)}, on which the results in this paper rely on. 
In \cite{Dias_Dinis(23)} two different classes of meadows have been studied: pre-meadows with $\abf$ and common meadows (introduced in \cite{Bergstra2015}). 

        A \emph{pre-meadow} is a structure $(P,+,-,\cdot)$ satisfying the following equations
    \begin{multicols}{2}
\begin{enumerate}
\item[$(P_1)$] $(x+y)+z=x+(y+z) $
\item[$(P_2)$] $x+y=y+x $ 
\item[$(P_3)$]  $x+0=x$ 
\item[$(P_4)$] $x+ (-x)=0 \cdot x$
\item[$(P_5)$] $(x \cdot y) \cdot z=x \cdot (y \cdot z)$ 
\item[$(P_6)$]  $x \cdot y=y \cdot x $
\item[$(P_7)$] $1 \cdot x=x$
\item[$(P_8)$] $x \cdot (y+z)= x \cdot y + x \cdot z$
\item[$(P_9)$] $-(-x)=x$
\item[$(P_{10})$] $0\cdot (x+y)=0\cdot x+0\cdot y$
\end{enumerate}
\end{multicols}

 Let $P$ be a pre-meadow and let $P_z:=\{x\in M\mid 0\cdot x=z\}$. We say that $P$ is a \emph{pre-meadow with $\abf$} if there exists a unique $z\in 0\cdot P$ such that $|P_z|=1$ (denoted by $\abf$) and $x+\abf=\abf$, for all $x\in P$. 
    A \emph{common meadow} is a pre-meadow with $\abf$ equipped with an inverse function $(\cdot)\inv$ satisfying
    \begin{multicols}{2}
\begin{enumerate}
\item[$(M_1)$] $x \cdot x^{-1}=1 + 0 \cdot x^{-1}$
\item[$(M_2)$]$(x \cdot y)^{-1} = x^{-1} \cdot y^{-1}$
\item[$(M_3)$] $(1 + 0 \cdot x)^{-1} = 1 + 0 \cdot x $
\item[$(M_4)$] $ 0^{-1}=\abf$
\end{enumerate}
\end{multicols}

    Examples \ref{E:partition_assoc_meadow} and \ref{E:smallestn} show that the classes of pre-meadows with $\abf$ and of common meadows are distinct but not disjoint.

  Let $P$ be a pre-meadow and $z,z'\in 0\cdot P$. We say that $z$ is \emph{less than or     equal to} $z'$, and write $z\leq z'$, if and only if $z\cdot z' = z$.

In \cite{Dias_Dinis(23)} it was shown that every pre-meadow with $\abf$ is associated with a particular type of lattice, a directed lattice of rings. We recall the definition and the result below.

    \begin{definition}\label{D:LatticeRings}
        A \emph{directed lattice} of rings $\Gamma$ over a countable lattice $L$ consists on a family of commutative rings $\Gamma_i$ indexed by $i\in L$, such that $\Gamma_i$ is a unital commutative ring for all $i\in L\setminus \min(L)$ and $\Gamma_{\min(L)}$ is the zero ring, together with a family of ring homomorphisms $f_{j,i}:\Gamma_i\rightarrow\Gamma_j$ whenever $i>j$ such that $f_{k,j}\circ f_{j,i}=f_{k,i} $ for all $i>j>k$. 
    \end{definition}

\begin{theorem}\label{T:Directedlat}   
Given a directed lattice of rings $\Gamma$ over the lattice $L$ there is an associated pre-meadow with $\abf$ defined by $M=\bigsqcup_{i\in L}\Gamma_i$. Additionally, $M$ is a common meadow if and only if for all $x\in\Gamma_i\subseteq M$ the set 

$$J_x=\{j\in I\mid f_{j,i}(x)\in \Gamma_j^{\times}\}$$
        has a unique maximal element.

Also, given  $M$  a pre-meadow with $\abf$, there is a directed lattice of rings over the lattice $0\cdot M$.
\end{theorem}

Note that, given a pre-meadow with $\abf$, with $0\cdot M$ finite and such that the partial order in $0\cdot M$ is a total order, we may conclude $M$ is a common meadow. 
Clearly, if $M$ is a finite pre-meadow with $\abf$ then both $0\cdot M$ and $M_{0\cdot z}$, for all $z\in M$, are finite sets.

Theorem~\ref{T:Directedlat} is constructive, as its proof shows explicitly how to construct a pre-meadow with $\abf$, once a directed lattice is given. Let us illustrate the process with an example.

\begin{exmp}\label{E:partition_assoc_meadow}
Consider the following lattice
\[\begin{tikzcd}
	& \z_6 \\
	& {} && {} \\
	\z_3 && \z_2 \\
	\\
	& {\{\abf\}}
	\arrow["\pi_1", from=1-2, to=3-1]
	\arrow["\pi_2"',  from=1-2, to=3-3]
	\arrow[from=3-3, to=5-2]
	\arrow[from=3-1, to=5-2]
\end{tikzcd}\]
where $\pi_1:\z_6\rightarrow \z_3$ and $\pi_2: \z_6 \rightarrow \z_2$ are the natural projection maps that send $1$ to $1$. Then, with the operations defined by the lattice, $M=\z_6\sqcup\z_3\sqcup\z_2\sqcup \{\abf\}$ is a pre-meadow with $\abf$ of order $12$. 
Additionally, simple calculations show that
\begin{multicols}{2}
    \begin{itemize}
    \item $J_{[0]_6}=\{\abf\}$
    \item $J_{[1]_6}=\{[0]_6,\abf\}$
    \item $J_{[2]_6}=\{[0]_3,\abf\}$
    \item $J_{[3]_6}=\{[0]_2,\abf\}$
    \item $J_{[4]_6}=\{[0]_3,\abf\}$
    \item $J_{[5]_6}=\{[0]_6,\abf\}$
\end{itemize}
\end{multicols}
and so $M$ is in fact a common meadow.

\end{exmp}

\section{Finite meadows}\label{S:Finite_meadows}


We start by showing that there exist common meadows of every order greater than or equal to $3$.

\begin{proposition}\label{P:existence}
    Let $n\geq 2$. Then there exists a common meadow of order $n+1$.
\end{proposition}

\begin{proof}
    Since $\z_n$ is a commutative ring, for $n \geq 2$ we can construct the pre-meadow with $\abf$
    \begin{center}
    \begin{tikzcd}
	{\z_n}\\
	{\{\abf\}}
        \arrow[from=1-1, to=2-1]
    \end{tikzcd}
    \end{center}
which clearly has $n+1$ elements. Since the order is total, it is in fact a common meadow.
\end{proof}

By Theorem \ref{T:Directedlat}, given  $M$, a pre-meadow with $\abf$, there exists a lattice associated with $M$, denoted by $0\cdot M$. The following proposition shows that there are infinitely many common meadows associated with each lattice.

\begin{proposition}\label{P:Lmn}
    Let $L$ be a lattice with $n+1$ vertices, where $n>0$, then for each natural number $k$ greater than or equal to $2$ there is a common meadow $M$ with $kn+1$ elements such that the lattice associated with $M$ is isomorphic to $L$.
\end{proposition}
\begin{proof}
    Given a lattice with $n+1$ vertices, where $n>0$, replace its minimum by $\{\abf\}$, all other vertices by the ring $\z_k$ and connect these rings by the identity function. By Theorem~\ref{T:Directedlat}, the resulting structure is a pre-meadow with $\abf$ which clearly has $kn+1$ elements. Since the transition maps are the identity, one easily sees that $J_x$ either has $x$ as the unique maximal element, or $J_x=\{\abf\}$, in which case the maximal element is also unique. Hence $M$ is a common meadow.
\end{proof}

\begin{exmp}\label{E:lattice8}
Consider the following lattice with $8$ elements
    \[\begin{tikzcd}
	& \bullet \\
	\bullet & \bullet & \bullet \\
	\bullet & \bullet & \bullet \\
	& \bullet
	\arrow[from=1-2, to=2-1]
	\arrow[from=1-2, to=2-3]
	\arrow[from=2-1, to=3-2]
	\arrow[from=2-3, to=3-2]
	\arrow[from=1-2, to=2-2]
	\arrow[from=2-2, to=3-1]
	\arrow[from=2-2, to=3-3]
	\arrow[from=3-3, to=4-2]
	\arrow[from=3-1, to=4-2]
	\arrow[from=3-2, to=4-2]
	\arrow[from=2-3, to=3-3]
	\arrow[from=2-1, to=3-1]
\end{tikzcd}\]

We can construct, for example, a common meadow with $9 \times 7 + 1=64$ elements as follows where, as in the proof of Proposition \ref{P:Lmn}, the arrows between the cyclic groups are the identity
\[\begin{tikzcd}
	& {\z_9} \\
	{\z_9} & {\z_9} & {\z_9} \\
	{\z_9} & {\z_9} & {\z_9} \\
	& {\{\abf\}} \\
	\arrow[from=1-2, to=2-1]
	\arrow[from=1-2, to=2-3]
	\arrow[from=2-1, to=3-2]
	\arrow[from=2-3, to=3-2]
	\arrow[from=1-2, to=2-2]
	\arrow[from=2-2, to=3-1]
	\arrow[from=2-2, to=3-3]
	\arrow[from=3-3, to=4-2]
	\arrow[from=3-1, to=4-2]
	\arrow[from=3-2, to=4-2]
	\arrow[from=2-3, to=3-3]
	\arrow[from=2-1, to=3-1]
\end{tikzcd}\]
\end{exmp}

By Proposition~\ref{P:Lmn} we obtain a (non-optimal) lower bound on the number of common meadows of odd order in terms of the number of lattices with $n$ elements, which we denote by $L(n)$.

\begin{proposition}\label{P:bound_odd}
    Let $n$ be a natural number greater than or equal  to $1$. Then there exist at least $L(n+1)$ common meadows with $2n+1$ elements.
\end{proposition}

As for the the number of common meadows of even order we can also obtain a (non-optimal) lower bound (see Figure~\ref{fig:lowerbound} for a list of these bounds for $n=11$ to $30$).

\begin{proposition}\label{P:bound_even}
     Let $n$ be a natural number greater than or equal  to $6$. Then there exist at least $L(n-5)$ common meadows with $2n$ elements.
\end{proposition}

\begin{proof}
    We have that $2n =6+3+2(n-5)+1$. Then, given a lattice $L$ with more than $5$ vertices, we can create a new lattice by adding three extra vertices $v_1,v_2,v_3$; connecting $v_1$ to the maximum of $L$, and becoming the new maximum; connecting $v_3$ to the minimum of $L$, and becoming the new minimum; and by creating a new edge between $v_1$ and $v_2$, and a new edge between $v_2$ and $v_3$, as illustrated below. 

    \[\begin{tikzcd}
	& {v_1} \\
	{v_2} && {\boxed{L}} \\
	& {v_3}
	\arrow[from=1-2, to=2-1]
	\arrow[from=2-1, to=3-2]
	\arrow[from=1-2, to=2-3]
	\arrow[from=2-3, to=3-2]
\end{tikzcd}\]

We can now construct a directed lattice by replacing $v_1$ with $\z_6$, $v_2$ with $\z_3$, $v_3$ with $\{\abf\}$ and all the vertices of $L$ with $\z_2$. The maps from $\z_6$ (which is, of course, isomorphic with $\z_3 \times \z_2$) to $\z_3$ and to $\z_2$ are the projection maps.

Similarly to Example \ref{E:partition_assoc_meadow}, for all $x\in \z_6$ we have that $J_x$ has a unique maximal element, and so it is in fact a common meadow.
\end{proof}

Since common meadows are pre-meadows with $\abf$ the bounds given in Propositions \ref{P:bound_odd} and \ref{P:bound_even} are trivially also bounds for pre-meadows with $\abf$.

\begin{definition}
   Given $M$, a pre-meadow with $\abf$,  with $n$ elements we say that a partition $a_0+\cdots+a_s+1$ of $n$ is \emph{associated} with $M$ if there exist $x_i\in 0\cdot M\setminus\{\abf\}$ such that for all $i\in\{1,\cdots,s\}$ we have $a_i=|M_{x_i}|$, and $M_{x_i}\neq M_{x_j}$ for $i\neq j$.
\end{definition}

As an example, observe that the partition $12=6+3+2+1$ is associated with the common meadow in Example \ref{E:partition_assoc_meadow}, and the partition $64=9+9+9+9+9+9+9+1$ is associated with the common meadow in Example \ref{E:lattice8}.

\begin{definition}
    Let $a_0+\cdots+a_s+1$ be a partition of a natural number $n$ greater than or equal to $3$, such that $a_i>1$ for $i\in\{0,\cdots,s\} $. We say that the partition is \emph{admissible} if there is some $k\in\{0,\cdots,s\}$ such that all prime divisors of $a_i$, for $i\in\{1,\cdots,s\}$,  are also prime divisors of $a_k$.
\end{definition}
Note that in admissible partitions the number $1$ occurs exactly once.
\begin{exmp}\label{E:5}
    There are $7$ partitions of $5$:
     \begin{multicols}{3}
    \begin{itemize}
    \item $5$
    \item $4+1$
    \item $3+2$
    \item $3+1+1$
    \item $2+2+1$
    \item $2 + 1 + 1 + 1$
    \item $1 + 1 + 1 + 1 + 1$
  \end{itemize}
     \end{multicols}
     Out of these, only the partitions $4+1$ and $2+2+1$ are admissible.
\end{exmp}

We now show that a partition is admissible if and only if there is a pre-meadow with $\abf$ associated with that partition. The proof requires the following lemma concerning the existence of ring homomorphisms. 

We denote by $\homr(S,T)$ the set of all unital ring homomorphisms from $R$ to $S$.

\begin{lemma}\label{L:RingMorphisms}
    Let $m=p_1^{\alpha_1}\cdots p_s^{\alpha_s}$ and $n=p_1^{\beta_1}\cdots p_t^{\beta_t}$ be positive natural numbers, wiith $p_i$ prime numbers for $i\in \{1,\cdots,\max\{s,t\}\}$. Then there exist unital rings $S,T$ of orders $m$ and $n$, repectively, such that $\homr(S,T)\neq \emptyset$ if and only if $t\leq s$.
\end{lemma}

\begin{proof}
    Assume first that $t\leq s$. Let $\z_p^n$ be the Cartesian product of $n$ copies of $\z_p$. We define $S:=\z_{p_1}^{\alpha_1}\times\cdots \times\z_{p_s}^{\alpha_s}$ and $T:=\z_{p_1}^{\beta_1}\times\cdots \times\z_{p_t}^{\beta_t}$. For each $i\in \{1,\cdots,t\}$, consider the ring homomorphism $\psi_{p_i}:\z_{p_i}^{\alpha_i}\rightarrow\z_{p_i}^{\beta_i}$ defined by $\psi_{p_i}= i_{p_i}\circ \, \pi_{p_i}$, where $\pi_{p_i}:\z_{p_i}^{\alpha_i}\rightarrow\z_{p_i}$ is the projection homomorphism onto the first coordinate, and $i_{p_i}:\z_{p_i}\rightarrow \z_{p_i}^{\beta_i}$  is the inclusion homomorphism $i_{p_i}(1)=(1,\cdots,1)$. Consider now the projection homomorphism onto the first $t$ coordinates
    $$\pi: S 
    \rightarrow \z_{p_1}^{\alpha_1}\times\cdots \times\z_{p_t}^{\alpha_t}.$$
    
    Then the composition $(\psi_{p_1}\times \cdots \times \psi_{p_t}) \circ \, \pi$, for $i\in \{1,\cdots,t\}$ is easily shown to be a ring homomorphism from $S$ to $T$.

    Assume now  that there exists a ring homomorphism $f:S\rightarrow T$ between the unital rings $S$ and $T$ of orders $m$ and $n$, respectively. Note that, from the first isomorphism theorem, the order of $\Ima(f)$ divides the order of $S$, and so $\Ima(f)$ is a ring of order $m'=p_1^{\alpha'_1}\cdots p_s^{\alpha'_s}$. Now, if all prime divisors of the order of $T$ are also prime divisors of the order of $\Ima(f)$, then they must also divide the order of $S$. Hence, we may assume that $S$ is a unital subring of $T$. But then the unit of $S$ is the same as the unit of $T$, which we denote by $1$. From the fact that $S$ is a ring of order $m=p_1^{\alpha_1}\cdots p_s^{\alpha_s}$ we have that $m\cdot 1 = 0$. Now let $T=T_{p_1}\times \cdots \times T_{p_t}$ be the decomposition of $T$ into maximal subrings of order a power of $p$, and consider the projection homomorphism onto the last coordinate $\pi_t:S\rightarrow S_{p_t}$. Since $\pi_t(1)$ is the identity of $S_{p_t}$ and $m\cdot 1 = 0$, we have that $m\cdot \pi_t(1) = 0$, and so $p_t$ must divide $m$, which implies that $t\leq s$.
\end{proof}

We will also require the following result from \cite{Dias_Dinis(23)}.

\begin{lemma}\label{P:Transitionmaps}
            Let $M$ be a pre-meadow with $\abf$. If $z,z'\in 0\cdot M$ are such that $z\leq z'$, then the map  $f_{z,z'}:M_{z'} \rightarrow M_z$ defined by $f_{z,z'}(x)=x+z$ is a ring homomorphism. Moreover, if $z,z',z''\in 0\cdot M$ are such that $z\leq z'\leq z''$, then $f_{z,z'}\circ f_{z',z''}=f_{z,z''}$. 
\end{lemma}

\begin{theorem}\label{T:Main}
    Let $n\geq 3$  and $p$ be a partition of $n$. Then $p$ is admissible if and only if there exists $M$, a pre-meadow with $\abf$,  such that $p$ is associated with $M$.
\end{theorem}

\begin{proof}
        Let $a_0+\cdots+a_s+1$ be an admissible  partition of $n$. We need to show that there exists $M$, a pre-meadow with $\abf$, associated with this partition. Without loss of generality we may assume that all prime divisors of $a_i$ with $i\in\{0,\cdots,s\}$ are also prime divisors of $a_0$. Let $a_0:=p_1^{\alpha_1}\cdots p_t^{\alpha_t}$ be the prime factorization of $a_0$, and $a_i=p_1^{\alpha_{i,1}}\cdots p_t^{\alpha_{i,t}}$ the prime factorization of $a_i$ with $\alpha_{i,j}\in \n$, where $i\in \{1,\cdots,s\}$ and $j\in\{1,\cdots, t\}$.

    From the first part of the proof of Lemma \ref{L:RingMorphisms}, there exist rings $R_0,\cdots,R_t$ such that for all $i\in\{1,\cdots,t\}$ there is a ring homomorphism $f_i:R_0\rightarrow R_i$. Then the set $M=\left(\bigsqcup_{i=0,\cdots,s} R_i\right)\bigsqcup \{\abf\}$ with the operations
    \begin{itemize}
        \item $x+y=\begin{cases}
                x+y,&\text{ if } x,y\in R_i\\
                x+f_i(y),&\text{ if } x\in R_i, y\in R_0\\
                \abf,& \text{otherwise}
        \end{cases}$
        \item $x\cdot y=\begin{cases}
                x\cdot y,& \text{ if } x,y\in R_i\\
                x\cdot f_i(y),& \text{ if } x\in R_i,y\in R_0\\
                \abf,& \text{otherwise} 
        \end{cases}$
    \end{itemize}
    is a pre-meadow with $\abf$. Moreover, for each $x\in 0\cdot M\setminus\{\abf\}$ we have that $M_x=R_i$, where $x\in R_i$. Hence $|M_x|=a_i$, by construction. That is, the partition $a_0+\cdots+a_s+1$ is associated with $M$.

    Assume now that $M$ is a pre-meadow with $\abf$ associated with the partition $a_0+\cdots+a_s+1$ of $n$. With a convenient rearrangement of the indices, we may assume that $|M_0|=a_0$. Let $x_i\in 0\cdot M$ be such that $|M_{x_i}|=a_i$. Then by Lemma \ref{P:Transitionmaps}, the function:
    \begin{align*}
        f_{x_i,0}:M_0&\rightarrow M_{x_i}\\
                x&\mapsto x+ x_i
    \end{align*}
    is a ring homomorphism. By Lemma \ref{L:RingMorphisms}, the prime divisors of $a_i$ are also prime divisors of $a_0$, and so the partition is admissible.
\end{proof}

\begin{exmp}
    Consider the admissible partition $41=30+5+3+2+1$. It can be shown that $41$ is the smallest $n$ that has an admissible partition not associated with a common meadow. In fact, since the only ring with $30$ elements is $\z_2\times\z_3\times\z_5$ the only pre-meadow with $\abf$ associated with this partition is 
    \[\begin{tikzcd}
	& {\z_2\times\z_3\times\z_5} \\
	{\z_2} & {\z_3} & {\z_5} \\
	& {\{\abf\}}
	\arrow["\pi_1",from=1-2, to=2-1]
	\arrow["\pi_2",from=1-2, to=2-2]
	\arrow["\pi_3",from=1-2, to=2-3]
	\arrow[from=2-1, to=3-2]
	\arrow[from=2-2, to=3-2]
	\arrow[from=2-3, to=3-2]
\end{tikzcd}\]
    which is not a common meadow since $J_{([0]_2,[1]_3,[1]_5)}=\{[0]_3,[0]_5,\abf\}$ has clearly two distinct maximal elements.
\end{exmp}

 \begin{corollary}\label{P:gcd}
      Let $n\in \n$ be a positive natural number. If $n$ cannot be written as a sum of natural numbers $\{a_i\}_{i\in I}$ greater than or equal to $2$ such that, for $i\in I$, $gcd(a_0,a_i)\neq 1$, then all common meadows with $n+1$ elements are rings with $\abf$. In particular, if $x\neq \abf$, then $0\cdot x = 0$.
\end{corollary}

We can use Corollary~\ref{P:gcd} to show that for $n=3,4,6,8$ there exists a unique common meadow of that order. These common meadows are listed in Figure~\ref{F:M3_4_6_8}.

    \begin{corollary}
    There exists exactly one common meadow of order $3,4,6$ or $8$.
    \end{corollary}

\begin{proof}
The admissible partitions of  $3,4,6$ or $8$ are, respectively, 
 \begin{multicols}{4}
    \begin{itemize}
    \item 3=2+1
    \item 4=3+1
    \item 6=5+1
    \item 8=7+1
\end{itemize}
\end{multicols}
By Corollary~\ref{P:gcd} there is only one common meadow for each of these orders: the common  meadows presented in Figure~\ref{F:M3_4_6_8}.
\end{proof}

\begin{figure}[H]
    \centering
    \begin{tikzcd}
	{\z_2} & {\z_3} & {\z_5} & {\z_7}\\
	{\{\abf\}} & {\{\abf\}} & {\{\abf\}} & {\{\abf\}}\\ 
        \arrow[from=1-1, to=2-1]
	\arrow[from=1-2, to=2-2]
	\arrow[from=1-3, to=2-3]
        \arrow[from=1-4, to=2-4]
\end{tikzcd}
    \caption{The only common meadows of order $3,4,6$ and $8$}
    \label{F:M3_4_6_8}
\end{figure}

So, the first interesting cases are when $n=5$ and $n=7$. These cases will be dealt with in the next section. 
\begin{exmp}\label{E:smallestn}
    The smallest pre-meadow with $\abf$ that is not a common meadow is the one defined by the lattice:
    \[\begin{tikzcd}
	& {\z_2\times\z_2} \\
	{\z_2} && {\z_2} \\
	& {\{\abf\}}
	\arrow["{\pi_1}"', from=1-2, to=2-1]
	\arrow["{\pi_1}", from=1-2, to=2-3]
	\arrow[from=2-3, to=3-2]
	\arrow[from=2-1, to=3-2]
\end{tikzcd}\]
    Since the set $J_{([1]_2,[0]_2)}$ has two distinct maximal elements it is indeed not a common meadow.
\end{exmp}

\section{Common meadows of orders 5 and 7}\label{S:5_and_7}

A simple observation is that, in order for a partition to represent a directed lattice it must contain exactly one occurrence of the term $1$, so that the trivial ring can be counted exactly once. Additionally, we recall that if a directed lattice represents a total order, or if all the maps are isomorphisms then the directed lattice is associated with a common meadow. 

\subsection{Common meadows of order 5}\label{S:5}

    As seen in Example~\ref{E:5}, the admissible partitions of $5$ are
     \begin{multicols}{2}
    \begin{itemize}
    \item $5=4 + 1$
    \item $5=2 + 2 + 1$
  \end{itemize}
     \end{multicols}
     These partitions correspond to the following directed lattices

\[\begin{tikzcd}
	\bullet & \bullet\\
	\bullet & \bullet \\
	& \bullet &
	\arrow[from=1-1, to=2-1]
	\arrow[from=1-2, to=2-2]
        \arrow[from=2-2, to=3-2]
\end{tikzcd}\]

Since, up to isomorphism, the only ring with two elements is $\z_2$, and since $\z_2$ is generated by $1$ we have that the only homomorphism from $\z_2$ to $\z_2$  is the identity. There are, in turn, $4$ rings with $4$ elements. So, there exist exactly the following $5$ common meadows with $5$ elements (note that the lattices define a total order and so they are in fact common meadows):
     \[\begin{tikzcd}
	{\z_4} & {\z_2} & {\z_2 \times \z_2} & {\f_4} 
 & {\mM_2}  \\
	{\{\abf\}} & {\z_2} & {\{\abf\}} & {\{\abf\}} & {\{\abf\}}\\
	& {\{\abf\}} && && 
        \arrow[from=1-1, to=2-1]
	\arrow[from=2-2, to=3-2]
        \arrow["\Id",from=1-2, to=2-2]
	\arrow[from=1-3, to=2-3]
        \arrow[from=1-4, to=2-4]
        \arrow[from=1-5, to=2-5]
\end{tikzcd}\]
where $\f_4$ is the unique field with four elements and $\mM_2=\left\{ \begin{bmatrix}
x & y \\
0 & x 
\end{bmatrix} \mid x,y\in \z_2\right\} $.

\subsection{Common meadows of order 7}\label{E:7}

    The admissible partitions of $7$ are
    \begin{multicols}{2}
    \begin{itemize}
        \item $7=6+1$
        \item $7=4+2+1$
        \item $7=3+3+1$
        \item $7=2+2+2+1$        
    \end{itemize}
    \end{multicols}
These partitions correspond to lattices of orders $2,3,3$ and $4$ respectively. The following diagram illustrates all the possible lattices with $2,3$ and $4$ elements:

\[\begin{tikzcd}
	&&&&& \bullet \\
	\bullet & \bullet && \bullet && \bullet \\
	\bullet & \bullet & \bullet && \bullet & \bullet \\
	& \bullet && \bullet && \bullet
	\arrow[from=2-2, to=3-2]
	\arrow[from=3-2, to=4-2]
	\arrow[from=2-4, to=3-3]
	\arrow[from=2-4, to=3-5]
	\arrow[from=3-5, to=4-4]
	\arrow[from=3-3, to=4-4]
	\arrow[from=2-1, to=3-1]
	\arrow[from=1-6, to=2-6]
	\arrow[from=2-6, to=3-6]
	\arrow[from=3-6, to=4-6]
\end{tikzcd}\]

We conclude that the lattice of a pre-meadow with $\abf$ with $7$ elements  must be isomorphic to one of these. In fact, since up to isomorphism the only ring with two elements is $\z_2$, with $3$ elements is $\z_3$ and with $6$ elements is $\z_6$ and since both $\z_3$ and $\z_2$ are generated by $1$ we have that the only homomorphism from $\z_2$ to $\z_2$ (and from $\z_3$ to $\z_3$) is the identity. Hence, the pre-meadow with $\abf$ with $7$ elements associated with the partitions $6+1$, $3+3+1$ and $2+2+2+1$ are the following 
\[\begin{tikzcd}
	&&&&& {\z_2} \\
	{\z_6} & {\z_3} && {\z_2} && {\z_2} \\
	{\{\abf\}} & {\z_3} & {\z_2} && {\z_2} & {\z_2} \\
	& {\{\abf\}} && {\{\abf\}} && {\{\abf\}}
	\arrow["\Id", from=2-2, to=3-2]
	\arrow[from=3-2, to=4-2]
	\arrow["\Id", from=2-4, to=3-3]
	\arrow["\Id"', from=2-4, to=3-5]
	\arrow[from=3-5, to=4-4]
	\arrow[from=3-3, to=4-4]
	\arrow["\Id", from=1-6, to=2-6]
	\arrow["\Id", from=2-6, to=3-6]
	\arrow[from=3-6, to=4-6]
	\arrow[from=2-1, to=3-1]
\end{tikzcd}\]

One can easily see that these are in fact common meadows.

We now turn to the pre-meadow with $\abf$ of order $7$ associated with the partition $4+2+1$. There are two possibilities: either $|M_0|=4$, or $|M_0|=2$. 

Since $\z_2$ is generated by $1$, there exists at most one homomorphism from $\z_2$ to a ring with $4$ elements. A ring homomorphism from $\z_2$ to $\z_4$ does not exist since they have a different  characteristic. Then the lattices associated with the partition $4+2+1$ such that $|M_0|=2$ are the following.
\[\begin{tikzcd}
	{\z_2} & {\z_2} & {\z_2} \\
	{\z_2\times\z_2} & {\f_4} & {\mM_2} \\
	{\{\abf\}} & {\{\abf\}} & {\{\abf\}}
	\arrow[from=2-1, to=3-1]
	\arrow[from=1-1, to=2-1]
	\arrow[from=1-2, to=2-2]
	\arrow[from=1-3, to=2-3]
	\arrow[from=2-2, to=3-2]
	\arrow[from=2-3, to=3-3]
\end{tikzcd}\]

Since the order on these lattices is total they are also common meadows.

We now turn to the case $|M_0|=4$. If $R_4$ is a ring with $4$ elements, then the homomorphisms from $R_4$ to $\z_2$ are in bijection with maximal ideals of $R_4$. 

Since $\f_4$ is a field, there is no homomorphism from $\f_4$ to $\z_2$. The maximal ideals of $\z_2\times \z_2$ are exactly $\{0\}\times \z_2$ and $\z_2\times \{0\}$, which corresponds to the projection homomorphisms. In the ring $\z_4$ the only maximal ideal is the ideal generated by $2\in \z_4$ which corresponds to the unique homomorphism from $\z_4$ to $\z_2$ that sends $1$ to $1$. Finaly, it is easy to see that in $\mM_2$ the unique maximal ideal is the one generated by the matrix $ \begin{bmatrix}
0 & 1 \\
0 & 0 
\end{bmatrix}  $, which corresponds to the homomorphism $ \begin{bmatrix}
x & y \\
0 & x 
\end{bmatrix}  \mapsto x$. 
We conclude that the lattices associated with pre-meadows with $\abf$ of order $7$ such that $|M_0|=4$ are

\[\begin{tikzcd}
	{\z_2\times \z_2} & {\z_2\times \z_2} & {\z_4} & {\mM_2} \\
	{\z_2} & {\z_2} & {\z_2} & {\z_2} \\
	{\{\abf\}} & {\{\abf\}} & {\{\abf\}} & {\{\abf\}}
	\arrow[from=2-1, to=3-1]
	\arrow["{\pi_1}", from=1-1, to=2-1]
	\arrow["{\pi_2}", from=1-2, to=2-2]
	\arrow[from=1-3, to=2-3]
	\arrow[from=2-2, to=3-2]
	\arrow[from=2-3, to=3-3]
	\arrow[from=2-4, to=3-4]
	\arrow[from=1-4, to=2-4]
\end{tikzcd}\]

Where the maps are the ones mentioned before. Again, since the order is total, theses are common meadows. All that there is left to do is to check if any of these common meadows are isomorphic. In fact, the common meadows associated with the first and second lattice are isomorphic.

Consider the ring isomorphism  $f:\z_2\times\z_2\rightarrow \z_2\times\z_2$ defined by $f(x,y)=(y,x)$. It is easy to see that the following diagram commutes
\[\begin{tikzcd}
	{\z_2\times \z_2} & {\z_2\times \z_2} \\
	{\z_2} & {\z_2} \\
	{\{\abf\}} & {\{\abf\}}
	\arrow[from=2-1, to=3-1]
	\arrow["{\pi_1}", from=1-1, to=2-1]
	\arrow["{\pi_2}", from=1-2, to=2-2]
	\arrow[from=2-2, to=3-2]
	\arrow["f", from=1-1, to=1-2]
	\arrow["\Id"{description}, from=2-1, to=2-2]
\end{tikzcd}\]

and that it defines a common meadow isomorphism (see \cite[Section 4.3]{Dias_Dinis(23)}). We conclude that there exist $10$ pairwise non-isomorphic pre-meadow with $\abf$ of order $7$, and they are all in fact common meadows.

\section{An algorithm to enumerate pre-meadows with $\abf$}\label{S:algorithm}

The first author developed a code in GAP 4 \cite{GAP4} that allows to count the number of pre-meadows with $\abf$ of a given order and construct them. The code is available at \cite{github}.
 With this code, one is able to obtain all finite pre-meadows with $\abf$ from order $3$ up to order $16$. Hence the code gives a computational solution for these orders. 
 The fact that we are restricted to these orders has to do with two limitations of the GAP 4 database. The first limitation is that GAP 4 only has stored the rings of order $15$ and lower, so the fact that the code presented relies on this list prevents us from constructing, for example, all the pre-meadows with $\abf$ of order $17$ that are related with the admissible partition $16+1$. A second limitation has to do with the fact that GAP 4 has stored all the finite semigroups, but only up to order $8$ \cite{smallsemi}. In particular, the idempotent commutative semigroups which are equivalent to lattices, are stored but again only up to order $8$. This means that, in this way, one is not able to construct pre-meadows with $\abf$ associated with admissible partitions of the form $a_0 + \cdots +a_s+ 1$ with $s+2>8$. We would like to point out, however, that these are not limitations of the code but limitations on the database of GAP 4 which can be solved by adding to the database of the software the list of rings of order greater than $15$ and lattices of order greater than $8$.

\subsection{An overview of the code}\label{S:code}

The code uses the following GAP 4 packages: the package Semigroups \cite{semi} for basic operations on semigroups, the package Smallsemi \cite{smallsemi} for the database of small semigroups (up to order $8$), and the package Digraphs \cite{Digraphs} in order to provide a visualization of the directed lattices.

The code has several functions that we define in order to construct all pre-meadows with $\abf$ of a given order. 

The main function is  \texttt{enumeration\_meadows} which constructs all the pre-meadows with $\abf$ of a given order. The way that this function works is similar to the construction on the examples in Section~\ref{S:5_and_7}. We give a simplified version of the function \texttt{enumeration\_meadows} and describe the auxiliary functions.

\pagebreak 
\begin{lstlisting}
	meadow = [];
	for part in Adm_Part(n) do
		for s in C_Idempotent(Size(part)) do
			for perm in SymmetricGroup(Size(part)) do
				if perm = () then
					Add(meadow, [part,s]);
				else
					permpart := Permuted(part,perm);
					if Isomorph(meadow,permpart,s) then
						Add(meadow, [permpart,s]);
	fi;fi;od;od;od;

	list_meadow := [];
	for m in meadow do
		edges := Semigroup_Edges(m[2]);
		list_rings:=Latt_Rings(m[1]);
		for m_ring in list_rings do
			if Morphism_Sort(m_ring) then
				for h_edges in Edges_Morphism(edges,m_ring) do
					if Check_Composition(edges,h_edges,m_ring) then
						isomorphism := true;
						dir_lat:=[m_ring,edges,h_edges];
						for dir_meadow in list_meadow do
							if Meadow_Isomorphism(dir_lat,dir_meadow)=false then
								isomorphism:= false;
							fi;
						od;
						if isomorphism then
							Add(meadow,[m_ring,edges,h_edges]);
	fi;fi;od;fi;od;od;

\end{lstlisting}

\begin{itemize}
    \item \texttt{Adm\_Part(n)}: accepts a natural number \texttt{n} and returns the list of admissible partitions of $n$.
    \item \texttt{C\_Idempotent(n)}: accepts a natural number \texttt{n} and returns the list of all commutative idempotent semigroups of order $n$.
    \item \texttt{Isomorph(meadow,permpart,s)}: accepts a list \texttt{meadow}, a list \texttt{permpart} and a matrix \texttt{s}, and checks if in the list \texttt{meadow} there are isomorphic lattices defined by \texttt{permpart} and \texttt{s}, returning \texttt{true} if there is no isomorphic lattice and \texttt{false} otherwise.
    \item \texttt{Semigroup\_Edges(s)}: accepts a commutative idempotent semigroup \texttt{s} and returns the order relation defined by the semigroup \texttt{s}, i.e.\ the order defined by $x\leq y$ if $x*y=x$.
    \item \texttt{Latt\_Rings(permpart)}: accepts a list of numbers \texttt{permpart} and returns all lists of rings where the order of the $i^{\rm th}$ ring is equal to the $i^{\rm th}$ element of \texttt{permpart}.
    \item \texttt{Egdges\_Morphism(edges,m\_ring)}: accepts a list \texttt{edges} consisting on pairs
    of numbers and a list \texttt{m\_ring} consisting on a list of rings, and returns all 
    possible labels of the pairs in \texttt{edges} with ring homomorphisms consistent with \texttt{m\_ring}.
    \item \texttt{Check\_Composition(edges,h\_edges,m\_ring)}: accepts a list \texttt{edges} consisting on pairs of numbers, a list \texttt{m\_ring} consisting on a list of rings and \texttt{h\_edges} a list of labels of \texttt{edges} by ring homomorphisms, and returns \texttt{true} if the lattice defined by the ring homomorphisms is a commutative diagram.
    \item \texttt{Meadow\_Isomorphism(dir\_lat,dir\_meadow)}: accepts two lists \texttt{dir\_lat} and \texttt{dir\_meadow} and returns \texttt{true} if the pre-meadows defined by the lists are isomorphic, and false otherwise.
    \item \texttt{enumeration\_meadows(n)}: accepts a natural number $\texttt{n}$ (between $3$ and $15$) and returns all the pre-meadows with $\abf$ of that order and the number of them.
\end{itemize}

\subsection{On the output of the code}

The code outputs the list \texttt{list\_meadow} with all the pre-meadows with $\abf$ of order $n$. The output is composed of three parts. The first part consists of a list of rings of the form $[r,n]$ denoting the $r^{\rm th}$ ring of order $n$ (see Figure~\ref{fig:ringsGAP4}).

\begin{figure}[H]
    \centering
\begin{tabular}{|c|c|c|c|c|c|c|c|c|c|}
\hline
Code: & [1,1] & [2,2]  & [4,3]  & [4,9]           & [4,10]           & [4,11]           & [5,2]  & [6,4]  & [7,2]  \\ \hline
Ring  & $\{\abf\}$   & $\z_2$ & $\z_4$ & $\mathcal{M}_2$ & $\z_2\times\z_2$ & $\mathbb{F}_{4}$ & $\z_5$ & $\z_6$ & $\z_7$ \\ \hline
\end{tabular}
    \caption{List of some rings stored in GAP 4}
    \label{fig:ringsGAP4}
\end{figure}

These rings, which can be obtained using the GAP function \texttt{SmallRing(r,n)}, label the vertices of the lattice.
The second part is a list of pairs $[i,j]$ that should be read as "there exists an edge from the vertex $i$ to the vertex $j$". The third part is a list of ring homomorphisms that label the edges defined by the second part of the output. Here, each function $[a,b,c]\to [x,y,z]$ should be read as the ring homomorphism defined on the generators $a,b$ and $c$ that sends $a$ to $x$, $b$ to $y$, and $c$ to $z$.

As an example, we present the fifth pre-meadow with $\abf$, out of $41$, of order $9$, i.e.\ the fifth element of the list \texttt{enumeration\_meadows(9)}:
\begin{itemize}
    \item $[ [ 1, 1 ], [ 2, 2 ], [ 2, 2 ], [ 4, 3 ] ]$
    \item $[ [ 2, 1 ], [ 3, 1 ], [ 4, 1 ], [ 4, 2 ], [ 4, 3 ] ]$
    \item $[ [ a ] \to [ 0*a ], [ a ] \to [ 0*a ], [ a ] \to [ 0*a ], [ a ] \to [ a ], [ a ] \to [ a ] ]$
    
\end{itemize}

In this case the lattice will have $4$ points, which is the number of elements of the first list of the output, and the first point is labelled by the zero ring that we are denoting by $\{\abf\}$  (which corresponds to the pair $[1,1]$), the second and third are both labelled by the ring $\z_2$ (which corresponds to the pair $[2,2]$), and the fourth is labelled by the ring $\z_4$ (which corresponds to the pair $[4,3]$). 

From the second part, we have an homomorphism from all vertices to the vertex $1$ (because we have the pairs $[2,1],[3,1],[4,1]$) and, additionally, we have an homomorphism from the vertex $4$ to the vertex $2$ and to the vertex $3$ (because we have the pairs $[4,2]$ and $[4,3]$). 

Finally, the homomorphism corresponding to the edge $[2,1]$ is denoted by $[ a ] \to [ 0*a ]$ which means that the unique generator of the ring $\z_2$ is sent to $0*a=0=\abf$, and similarly for the edges $[3,1]$ and $[4,1]$. The homomorphism corresponding to the edge $[4,2]$ is denoted by $[ a ] \to [ a ]$ which means that it is the only homomorphism that sends the unit of $\z_4$ to the unit of $\z_2$. For the edge $[4,3]$ the reasoning is similar. 

\begin{figure}
\scriptsize
    \centering
    \[\begin{tikzcd}
	& 4 &&&& {[4,3]} &&&& {\mathbb{Z}_4} \\
	2 && 3 && {[2,2]} && {[2,2]} && {\mathbb{Z}_2} && {\mathbb{Z}_2} \\
	& 1 &&&& {[1,1]} &&&& {\{\abf\}}
	\arrow[from=1-2, to=2-1]
	\arrow[from=1-2, to=2-3]
	\arrow[from=2-3, to=3-2]
	\arrow[from=2-1, to=3-2]
	\arrow["{[ a ] \to [ a ]}"', from=1-6, to=2-5]
	\arrow["{[ a ] \to [ a ]}", from=1-6, to=2-7]
	\arrow["{[ a ] \to [ 0*a ]}"', from=2-5, to=3-6]
	\arrow["{[ a ] \to [ 0*a ]}", from=2-7, to=3-6]
	\arrow[from=2-9, to=3-10]
	\arrow[from=2-11, to=3-10]
	\arrow["{1\mapsto 1}"', from=1-10, to=2-9]
	\arrow["{1\mapsto 1}", from=1-10, to=2-11]
\end{tikzcd}\]
    \caption{From the output to common meadows}
    \label{fig:output}
\end{figure}

In Figure \ref{fig:output} we can see a schematic diagram of the whole process. We have, on the left, the lattice given by the first part of the code's output and, in the center. the lattice obtained after the entire output of the code is added. Using the GAP 4 function \texttt{SmallRing(r,n)} to get the $r^{\rm th}$ ring of order $n$ we obtain the lattice on the right.

In Figure~\ref{Table} we present a summary of the number of lattices, commutative rings, pre-meadows with $\abf$ and admissible partitions for values of $n$ from $3$ up to $16$. The number of pre-meadows with $\abf$  is also plotted in Figure~\ref{fig:graph}. Figure \ref{fig:lowerbound} presents some lower bounds for $n=11$ to $30$.

\begin{figure}
    \centering
\begin{tabular}{|c|c|c|c|c|}
\hline
\# & Lattices   & C. Rings & Pre-Meadows w/ $\abf$ & Adm. Part. \\ \hline
3  & 1          & 1        & 1       & 1          \\ \hline
4  & 2          & 4        & 1       & 1          \\ \hline
5  & 5          & 1        & 5       & 2          \\ \hline
6  & 15         & 1        & 1       & 1          \\ \hline
7  & 53         & 1        & 10      & 4          \\ \hline
8  & 222         & 10       & 1       & 1          \\ \hline
9  & 1078       & 4        & 41      & 5          \\ \hline
10 & 5994       & 1        & 7       & 3          \\ \hline
11 & 37622      & 1        & 122     & 8          \\ \hline
12 & 262776     & 1        & 2       & 2          \\ \hline
13 & 2018305    & 4        & 552     & 14         \\ \hline
14 & 16873364   & 1        & 6       & 3          \\ \hline
15 & 152233518  & 1        & 2355    & 17         \\ \hline
16 & 1471613387 & 37       & 60      & 11         \\ \hline
\end{tabular}
    \caption{Comparison of the number of lattices, commutative rings, pre-meadows with $\abf$ and admissible partitions from $n=3$ to $16$.}
    \label{Table}
\end{figure}

\begin{figure*}
    \begin{center}
    \includegraphics[width=\textwidth]{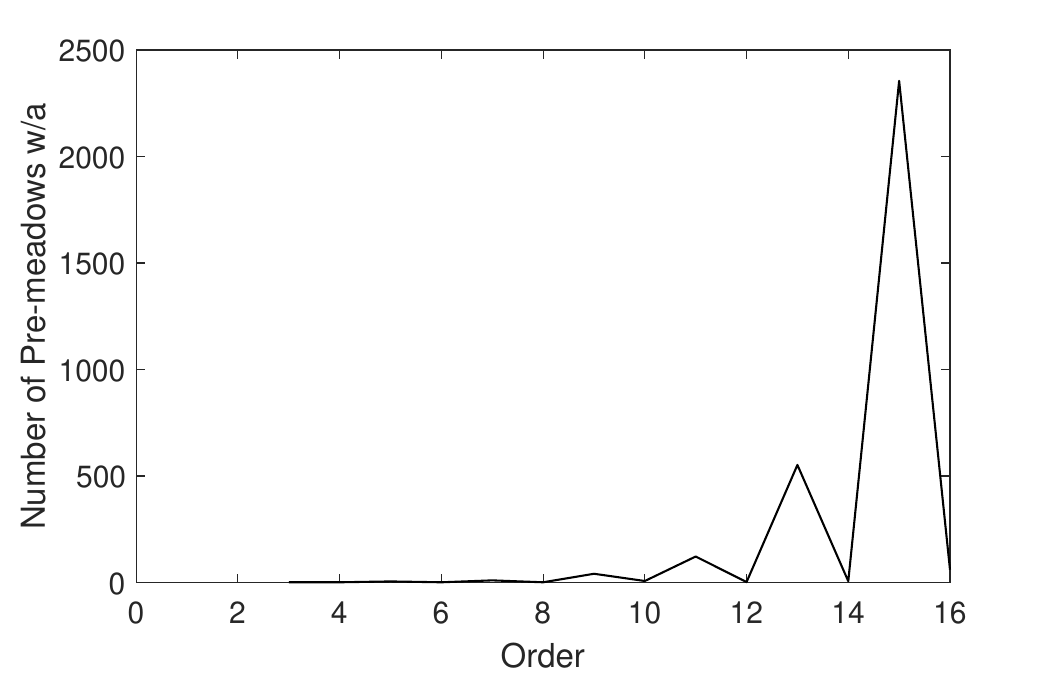}
    \end{center}
    \caption{Number of pre-meadows with $\abf$}    
    \label{fig:graph}
\end{figure*}

\begin{figure}
\begin{center}
\begin{tabular}{|c|c|| c | c|}
\hline
Order & Lower Bound & Order & Lower Bound \\ \hline
11    & 15 &21    & 37622            \\ \hline
12    & 1 & 22    & 15               \\ \hline
13    & 53 & 23    & 262776         \\ \hline
14    & 1 & 24    & 53           \\ \hline
15    & 222  & 25    & 2018305           \\ \hline
16    & 1 & 26    & 222              \\ \hline
17    & 1078  & 27    & 16873364          \\ \hline
18    & 2    & 28    & 1078          \\ \hline
19    & 5994  & 29    & 1471613387      \\ \hline
20    & 5    & 30    & 5994        \\ \hline

\end{tabular}
    \end{center}
    \caption{Lower bounds for the number of common meadows of orders $12$ to $30$}    
    \label{fig:lowerbound}
\end{figure}

\section{Final remarks}\label{S:finalremarks}

The lower bounds given by Propositions \ref{P:bound_odd} and \ref{P:bound_even}, illustrated in Figure~\ref{fig:lowerbound}, show that for both odd and even order, the number of  pre-meadows with $\abf$ (and of common meadows) have an exponential growth albeit at different rates. However, our bounds are very far from being optimal, as can be seen in Figure \ref{Table}. There are indeed some obvious ways to improve them. For example, the bound for odd order was established by observing that one has the obvious partition $$2n+1=\underbrace{2+\cdots+2}_{n \text{ times}}+1.$$
Now, of course one can always group some $2$'s into a $4$, say the first two, thus obtaining new admissible partitions which give rise to pre-meadows with $\abf$. A similar but more involved technique also allows to refine the lower bounds for even order. However, so far we do not know if it is possible to obtain \emph{optimal} lower bounds. The same question can be posed for upper bounds.

The "saw" shape of the graph given in Figure \ref{fig:graph} suggests that the number of pre-meadows with $\abf$  of order $2k+1$ is always greater than the number of pre-meadows with $\abf$  of order $2k+2$, for $k>2$. For the moment, a proof of such result escapes us. 

Our ultimate goal is to obtain a function enumerating all finite common meadows. 
Theorem \ref{T:Main} gives a relation between some partitions of $n$ and pre-meadows with $\abf$, which suggests a connection with combinatorics/number theory. Additionally, a numerical condition for an admissible partition to be associated with a common meadow could be enough to be able to enumerate common meadows. Of course, another possible route to solve (computationally) the problem is via a refinement of the code given in Section \ref{S:code}. 

\subsection*{Acknowledgments}

Both authors acknowledge the support of FCT - Funda\c{c}\~ao para a Ci\^{e}ncia e Tecnologia under the project: 10.54499/UIDB/04674/2020, and the research center CIMA -- Centro de Investigação em Matemática e Aplicações. 

The second author also acknowledges the support of CMAFcIO -- Centro de Matem\'{a}tica, Aplica\c{c}\~{o}es Fundamentais e Investiga\c{c}\~{a}o Operacional under the project UIDP/04561/2020.

        \bibliographystyle{siam}
\bibliography{References}

\end{document}